\documentclass[a4paper,11pt]{article}
\usepackage[utf8]{inputenc}
\usepackage[T1]{fontenc}
\usepackage[english]{babel}
\usepackage{amsmath,amssymb,amsthm}
\usepackage[short]{optidef}				% Package for linear programming
\usepackage[shortlabels]{enumitem} 
\usepackage{comment}
\usepackage{standalone}
\usepackage{mathtools}
\usepackage{fullpage}
\usepackage[dvipsnames]{xcolor}
\usepackage{float}
\usepackage{mathdots}
\usepackage{tikz,subcaption}
\usepackage[subrefformat=parens,labelformat=parens]{subcaption}

\graphicspath{{./Figuras/}} %%%% path for figures
\usetikzlibrary{graphs,graphs.standard}
\usepackage{hyperref}
\usepackage{thmtools}
\usepackage{thm-restate}
\tikzstyle{line} = [draw, thick]

\usepackage{graphicx}

\usepackage{tikz}
\usetikzlibrary{calc,positioning,decorations.pathmorphing,decorations.pathreplacing}
\providecommand{\tikzsetnextfilename}[1]{}

\tikzstyle{selected edge} = [draw,line width=5pt,-,red!50]
\tikzstyle{vertex}=[circle,draw,fill=white,minimum size=10pt,inner sep=2pt]
\tikzstyle{black vertex}=[circle,draw,fill=black,minimum size=10pt,inner sep=2pt]
\tikzstyle{red vertex}=[circle,draw,fill=red!90,minimum size=10pt,inner sep=2pt]
\tikzstyle{gray vertex}=[circle,draw,fill=gray!90,minimum size=20pt,inner sep=1pt]
\tikzstyle{trivial}=[circle,draw,fill=black,minimum size=7pt,inner sep=1pt]
\tikzstyle{dummy}=[circle,minimum size=5pt,inner sep=0pt]
\tikzstyle{red edge} = [draw,very thick,-,red!80,decorate,decoration={snake, amplitude=.8pt, segment length=4pt}]
\tikzstyle{purple edge} = [draw,very thick,-,magenta]
\tikzstyle{blue edge} = [draw,very thick,-,blue!60]
\tikzstyle{green edge} = [draw,very thick,-,green!90]
\tikzstyle{brown edge} = [draw,thick,-,brown!90]
\tikzstyle{black edge} = [draw,thick,-,black!90]
\tikzstyle{dot} = [circle,inner sep=1pt,fill,name=#1]
\tikzstyle{extended line} = [shorten >=-#1,shorten <=-#1]
\usetikzlibrary{arrows.meta}

\definecolor{TikzVertex0}{HTML}{ec706b}
\definecolor{TikzVertex1}{HTML}{f09033}
\definecolor{TikzVertex2}{HTML}{f7d453}
\definecolor{TikzVertex3}{HTML}{dd74fa}
\definecolor{TikzVertex4}{HTML}{99eeae}
\definecolor{TikzVertex5}{HTML}{e9415a}
\definecolor{TikzVertex6}{HTML}{4fc55f}
\definecolor{TikzVertex7}{HTML}{45b5d7}
\definecolor{TikzVertex8}{HTML}{2867a4}
\definecolor{TikzVertex9}{HTML}{7918d3}

\definecolor{myblue}{RGB}{80,80,160}
\definecolor{mygreen}{RGB}{80,160,80}
\definecolor{myred}{RGB}{255,0,0}
\definecolor{mybrown}{RGB}{139,69,19}

\newtheorem{theorem}             {Theorem}
\newtheorem{lemma}     	[theorem] {Lemma}

\newtheorem{proposition}[theorem] {Proposition}

\newtheorem{observation}[theorem] {Observation}
   
\newtheorem{corollary}	[theorem] {Corollary}

\newtheoremstyle{case}{}{}{}{}{\bfseries}{:}{ }{}
\theoremstyle{case}

\numberwithin{subcase}{case}

\begin{document}

\title{Dynamic Coloring and Graph Squares of Regular Graphs}

\author{
Juan Gutierrez\thanks{This work was partially supported by Fondo Semilla UTEC 2025.}
\and
Grover Ugarte\footnotemark[1]
}
\date{}
\maketitle

\begin{center}
\footnotesize
Department of Computer Science\\
University of Engineering and Technology (UTEC), Lima, Peru\\
E-mail: \texttt{\{jgutierreza,grover.ugarte\}@utec.edu.pe}
\end{center}

\begin{abstract}
An $r$-dynamic coloring of a graph $G$ is a proper vertex coloring in
which every vertex $v$ sees at least $\min\{r,d(v)\}$ distinct colors
in its neighborhood. 
The minimum number of colors in such a
coloring is the $r$-dynamic chromatic number $\chi_r^d(G)$.
We study dynamic colorings of regular graphs. A straighforward observation shows that  $\chi_r^d(G)=\chi(G^2)$, for any $r$-regular graph $G$.
We prove that
$
\chi_3^d(G)=\chi(G^2)\le6
$
for every claw-free cubic graph $G$, and the bound is sharp.
For Hamiltonian claw-free cubic graphs, we improve the bound to $5$
apart from four explicit exceptions.
We also determine exactly the $r$-dynamic chromatic number of the
$4$-regular circulant graph $C_p(1,3)$ for each
$r\in\{2,3,4\}$. In particular, the case $r=4$ determines
$\chi(C^2_p(1,3))$.
\end{abstract}
\section{Introduction}

Throughout the paper, all graphs are finite, simple, and undirected,
and we use standard terminology and notation
\cite{bondy_graph_2008,diestel_graph_2017}.
For a vertex $v$ of a graph $G$, let $N_G(v)$ and $d_G(v)$ denote its
neighborhood and degree, respectively.
For a positive integer $r$, an \emph{$r$-dynamic coloring} of a graph
$G$ is a proper vertex coloring $c$ such that
\[
|c(N_G(v))|\ge \min\{r,d_G(v)\}
\]
for every $v\in V(G)$. The minimum number of colors in such a coloring
is the \emph{$r$-dynamic chromatic number} of $G$, denoted by
$\chi_r^d(G)$. Dynamic coloring, corresponding to $r=2$, was introduced
by Montgomery~\cite{montgomery2001dynamic}. Higher-order versions have
subsequently been studied under the names conditional coloring,
$r$-hued coloring, and $r$-dynamic coloring; see
\cite{CHEN202224,LAI20061997}.

Dynamic coloring provides a natural interpolation between ordinary
coloring and coloring the square of a graph. Recall that $G^2$ has
vertex set $V(G)$, with two distinct vertices adjacent whenever their
distance in $G$ is at most two. For every graph $G$,
\begin{equation}\label{eq:absolute-bounds}
\chi(G)=\chi_1^d(G)
\le \chi_2^d(G)
\le\cdots\le
\chi_{\Delta(G)}^d(G)
=\chi(G^2).
\end{equation}
In particular, if $G$ is $d$-regular, then
\begin{equation}\label{eq:regular-square}
\chi_r^d(G)=\chi(G^2).
\end{equation}
Thus, on regular graphs, the last member of the dynamic-coloring
hierarchy is precisely the classical square-coloring problem.

Both dynamic coloring and coloring graph squares have received
considerable attention for graphs of bounded degree. We write
$\text{ch}(G)$ for the list chromatic number of $G$.
Cranston and Kim~\cite{CranstonKim2008} proved that every connected
subcubic graph other than the Petersen graph satisfies
$
\text{ch}(G^2)\le8,
$
and that the bound is best possible. For planar cubic graphs,
Thomassen~\cite{Thomassen2018} proved the sharp bound
$
\chi(G^2)\le7.
$
Square coloring has also been studied under structural restrictions.
In particular, Cames van Batenburg and Kang
\cite{CamesVanBatenburgKang2019} proved that every claw-free graph
with clique number three satisfies
$
\chi(G^2)\le10,
$
and further bounds for squares of claw-free graphs were obtained in
\cite{DeJoannisKangPastor2019}.

The regular setting is also natural from the dynamic-coloring
perspective. Montgomery conjectured that every regular graph satisfies
$
\chi_2^d(G)\le\chi(G)+2.
$
Bowler et al.~\cite{BowlerEtAl2017} disproved this conjecture by showing
that for every integer $k\ge2$ there exists a regular graph $G$ with
$
\chi(G)=k
\qquad\text{and}\qquad
\chi_2^d(G)=2k.
$
Thus, regularity alone does not force the dynamic chromatic number to
remain close to the ordinary chromatic number, motivating the study of
structured classes of regular graphs.

Our first main result concerns claw-free cubic graphs. For such graphs,
the $3$-dynamic coloring problem is exactly the square-coloring problem
by~\eqref{eq:regular-square}. Li and Lai~\cite{LiLai2017} proved that
every claw-free graph satisfies
$
\chi_3^d(G)\le \max\{\chi(G)+3,7\},
$
which yields the upper bound $7$ for claw-free cubic graphs. We improve
this bound to $6$, and the improvement is best possible.

\begin{itemize}
    \item 
If $G$ is a claw-free cubic graph, then
$
\chi(G^2)=\chi_3^d(G)\le6.
$
Moreover, the bound is sharp. (Theorem \ref{thm:clawfreeleq6}).
\end{itemize}
%\begin{theorem}\label{thm:clawfree-cubic}
%If $G$ is a claw-free cubic graph, then
%$
%\chi(G^2)=\chi_3^d(G)\le6.
%$
%Moreover, the bound is sharp.
%\end{theorem}

Indeed, the triangular prism $C_3\square K_2$ has diameter two, and
hence its square is $K_6$.
The proof of Theorem \ref{thm:clawfreeleq6} combines the structure of claw-free cubic
graphs with incidence colorings of subcubic multigraphs. We then exploit
the structure more strongly in the Hamiltonian case and reduce the
bound by one, apart from four explicit graphs (to be defined later).

%\begin{theorem}\label{thm:clawfree-hamiltonian}
%Let $G$ be a Hamiltonian claw-free cubic graph. If $ G\notin \{C_3\square K_2,X_{12},X_{16}^{1},X_{16}^{2}\}, $ then
%$
%\chi(G^2)=\chi_3^d(G)\le5.
%$
%\end{theorem}

\begin{itemize}
    \item If $G$ is a Hamiltonian claw-free cubic graph then
$
\chi(G^2)=\chi_3^d(G)\le5$, except
for four particular graphs. (Theorem \ref{thm:ham-clawfree}).
\end{itemize}

For completeness, observe the case $r=2$ is immediate for claw-free cubic graphs. The neighborhood of every vertex contains an edge, and hence
every proper coloring is automatically $2$-dynamic. Consequently,
$
\chi_2^d(G)=\chi(G).
$

Our second direction is to determine the complete dynamic-coloring
hierarchy for a natural family of $4$-regular graphs. For positive
integers $p,n_1,\ldots,n_k$, the circulant graph
$C_p(n_1,\ldots,n_k)$ has vertex set $\mathbb Z_p$, with $i$ adjacent
to $i\pm n_j$ for every $j$. We focus on $C_p(1,3)$ and determine
$\chi_r^d(C_p(1,3))$ for every $r\in\{2,3,4\}$.

The following hold (Theorems
\ref{thm:circulant-r2},     
\ref{thm:circulant-r3}, and \ref{thm:circulant-r4}):
\begin{itemize}
 \item If $p\ge6$, then
\[
\chi_2^d(C_p(1,3))=
\begin{cases}
3,&9\mid p,\\
4,&\text{otherwise}.
\end{cases}
\]

\item If $p\ge12$, then
\[
\chi_3^d(C_p(1,3))=
\begin{cases}
4,&12\mid p,\\
5,&\text{otherwise}.
\end{cases}
\]

\item If $p\ge12$, then
\[
\chi_4^d(C_p(1,3))=
\begin{cases}
5,&5\mid p,\\
8,&p=16,\\
7,&p\in\{13,14,19,21,26,31\},\\
6,&\text{otherwise}.
\end{cases}
\]
\end{itemize}

Since $C_p(1,3)$ is $4$-regular for $p\ge12$, the last part also
determines the chromatic number of its square. In fact,
$
C^2_p(1,3)=C_p(1,2,3,4,6).
$

%Thus, the two parts of the paper address complementary aspects of dynamic coloring in low-degree regular graphs: sharp square-coloring bounds for a structural class of cubic graphs, and an exact determination of the entire dynamic-coloring hierarchy for the circulant family $C_p(1,3)$.

\section{Claw-free cubic graphs}

In this section we study $3$-dynamic colorings of claw-free cubic graphs.
A \emph{diamond} is an induced subgraph isomorphic to $K_4 - e$. 
We say that two diamonds are \emph{adjacent} if there exists an edge joining a degree-$2$ vertex of one diamond to a degree-$2$ vertex of the other. 
Contracting each diamond to a single vertex and preserving these adjacencies yields an auxiliary graph; if this graph is a path, we call the corresponding subgraph a \emph{string of diamonds}, and if it is a cycle, we call it a \emph{ring of diamonds}.

We have the next observation, whose proof is direct by inspection
(Figure~\ref{fig:cubic-claw-free-ring-string}a).

\begin{observation}\label{obs:ringX3leq5}
    If $G$ is a ring of diamonds then $\chi^d_3(G) \leq 5$.
\end{observation}

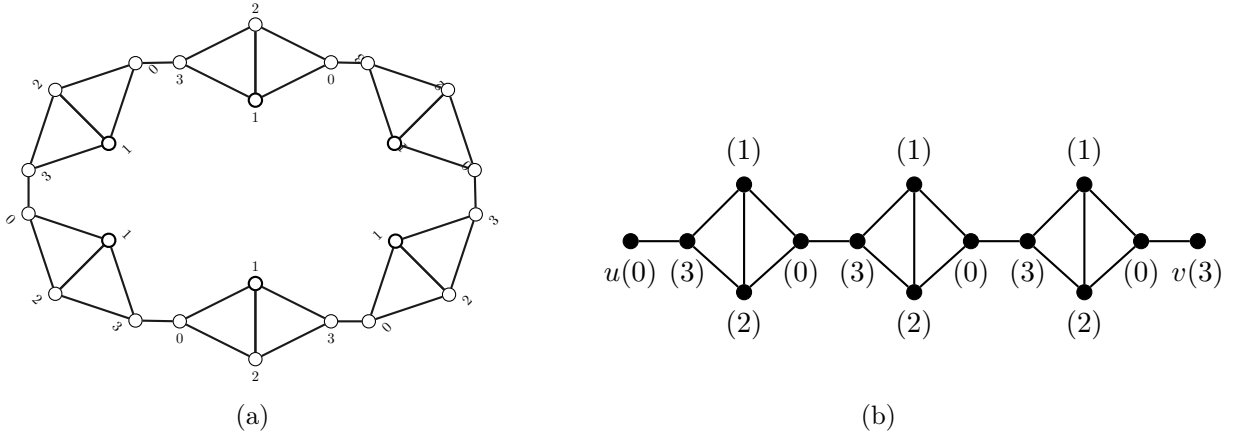
\begin{figure}[H]
    \centering
    \setcounter{subfigure}{0}

    \begin{subfigure}[b]{0.48\textwidth}
        \centering
        % \tikzexternaldisable

\tikzsetnextfilename{cubic-claw-free-ring}
\begin{tikzpicture}[scale=0.5, transform shape]
\centering
    \begin{scope}[yshift=1cm]
    \node[vertex,label=below:$0$] (A0) {};
    \node[vertex,label=below:$3$] (A1) at ($(A0)+(4,0)$) {};
    \node[vertex, draw=black, thick,label=above:$1$] (A2) at ($(A0)+(2,1)$) {};
    \node[vertex,label=below:$2$] (A3) at ($(A0)+(2,-1)$) {};
    \foreach \source/\dest in {A0/A2,A0/A3,A1/A2,A1/A3,A2/A3}
        \path[black edge] (\source) -- (\dest);
    \end{scope}

    \draw[black edge] (A2) -- (A3);

    \begin{scope}[xshift=5cm,yshift=1cm,rotate=45]
    \node[vertex,label=below:$0$] (B0) {};
    \node[vertex,label=below:$3$] (B1) at ($(B0)+(4,0)$) {};
    \node[vertex,draw=black,thick,label=above:$1$] (B2) at ($(B0)+(2,1)$) {};
    \node[vertex,label=below:$2$] (B3) at ($(B0)+(2,-1)$) {};
    \foreach \source/\dest in {B0/B2,B0/B3,B1/B2,B1/B3,B2/B3}
        \path[black edge] (\source) -- (\dest);
    \end{scope}

    \draw[black edge] (B2) -- (B3);

    \begin{scope}[xshift=7.8cm,yshift=5cm,rotate=135]
    \node[vertex,label=below:$0$] (C0) {};
    \node[vertex,label=below:$3$] (C1) at ($(C0)+(4,0)$) {};
    \node[vertex,draw=black,thick,label=above:$1$] (C2) at ($(C0)+(2,1)$) {};
    \node[vertex,label=below:$2$] (C3) at ($(C0)+(2,-1)$) {};
    \foreach \source/\dest in {C0/C2,C0/C3,C1/C2,C1/C3,C2/C3}
        \path[black edge] (\source) -- (\dest);
    \end{scope}

    \draw[black edge] (C2) -- (C3);

    \begin{scope}[yshift=7.85cm]
    \node[vertex,label=below:$3$] (D0) {};
    \node[vertex,label=below:$0$] (D1) at ($(D0)+(4,0)$) {};
    \node[vertex,label=above:$2$] (D2) at ($(D0)+(2,1)$) {};
    \node[vertex,draw=black,thick,label=below:$1$] (D3) at ($(D0)+(2,-1)$) {};
    \foreach \source/\dest in {D0/D2,D0/D3,D1/D2,D1/D3,D2/D3}
        \path[black edge] (\source) -- (\dest);
    \end{scope}

    \draw[black edge] (D2) -- (D3);

    \begin{scope}[xshift=-4cm,yshift=3.85cm,rotate=-45]
    \node[vertex,label=below:$0$] (F0) {};
    \node[vertex,label=below:$3$] (F1) at ($(F0)+(4,0)$) {};
    \node[vertex,draw=black,thick,label=above:$1$] (F2) at ($(F0)+(2,1)$) {};
    \node[vertex,label=below:$2$] (F3) at ($(F0)+(2,-1)$) {};
    \foreach \source/\dest in {F0/F2,F0/F3,F1/F2,F1/F3,F2/F3}
        \path[black edge] (\source) -- (\dest);
    \end{scope}

    \draw[black edge] (F2) -- (F3);

    \begin{scope}[xshift=-4cm,yshift=5cm,rotate=45]
    \node[vertex,label=below:$3$] (E0) {};
    \node[vertex,label=below:$0$] (E1) at ($(E0)+(4,0)$) {};
    \node[vertex,label=above:$2$] (E2) at ($(E0)+(2,1)$) {};
    \node[vertex,draw=black,thick,label=below:$1$] (E3) at ($(E0)+(2,-1)$) {};
    \foreach \source/\dest in {E0/E2,E0/E3,E1/E2,E1/E3,E2/E3}
        \path[black edge] (\source) -- (\dest);
    \end{scope}

    \draw[black edge] (E2) -- (E3);

    % Aristas entre los diamantes
    \draw[black edge] (A0) -- (F1);
    \draw[black edge] (D0) -- (E1);
    \draw[black edge] (E0) -- (F0);
    \draw[black edge] (C1) -- (D1);
    \draw[black edge] (B1) -- (C0);
    \draw[black edge] (A1) -- (B0);

\end{tikzpicture}
        \caption{$ $}
        \label{fig:cubic-claw-free-ring}
    \end{subfigure}
    \hfill
    \begin{subfigure}[b]{0.48\textwidth}
        \centering
        % \tikzexternaldisable

\tikzsetnextfilename{fig-cubic-string}
\begin{tikzpicture}[scale=0.75,
    vertex/.style={circle, draw, fill=black, inner sep=2pt},
    redvertex/.style={circle, draw, fill=red, inner sep=2pt},
    edge/.style={draw, thick}
]

% (mismo dibujo)
\node[vertex, label=below:$u (0)$] (u) at (0,0) {};
\node[vertex,label=below:$(3)$] (a1) at (1,0) {};
\draw[edge] (u) -- (a1);

\node[vertex, label=above:$(1)$] (t1) at (2,1) {};
\node[vertex,label=below:$(2)$] (b1) at (2,-0.9) {};
\node[vertex,label=below:$(0)$] (c1) at (3,0) {};
\draw[edge] (a1) -- (t1);
\draw[edge] (a1) -- (b1);
\draw[edge] (b1) -- (c1);
\draw[edge] (c1) -- (t1);
\draw[edge] (t1) -- (b1);

\node[vertex,label=below:$(3)$] (a2) at (4,0) {};
\draw[edge] (c1) -- (a2);

\node[vertex, label=above:$(1)$] (t2) at (5,1) {};
\node[vertex,label=below:$(2)$] (b2) at (5,-0.9) {};
\node[vertex,label=below:$(0)$] (c2) at (6,0) {};
\draw[edge] (a2) -- (t2);
\draw[edge] (a2) -- (b2);
\draw[edge] (b2) -- (c2);
\draw[edge] (c2) -- (t2);
\draw[edge] (t2) -- (b2);

\node[vertex,label=below:$(3)$] (a3) at (7,0) {};
\draw[edge] (c2) -- (a3);

\node[vertex, label=above:$(1)$] (t3) at (8,1) {};
\node[vertex, label=below:$(2)$] (b3) at (8,-0.9) {};
\node[vertex, label=below:$(0)$] (c3) at (9,0) {};
\draw[edge] (a3) -- (t3);
\draw[edge] (a3) -- (b3);
\draw[edge] (b3) -- (c3);
\draw[edge] (c3) -- (t3);
\draw[edge] (t3) -- (b3);

\node[vertex, label=below:$v (3)$] (v) at (10,0) {};
\draw[edge] (c3) -- (v);
\end{tikzpicture}
% \caption{$ $.}
        \caption{$ $}
        \label{fig:cubic-claw-free-string}
    \end{subfigure}

    \caption{A ring of diamonds (a) and a string of diamonds (b). In both cases, the assigned colors apear in parenthesis.}
    \label{fig:cubic-claw-free-ring-string}
\end{figure}

\begin{comment}
\begin{proposition}\label{prop:chi3Gleqchi3F}
    Let $F$ be a graph with $\Delta(F)=3$ and let $G$ be the graph obtained by replacing an edge $uv$ of $F$ with a string of diamonds.
    Then $\chi^d_3(G)  \leq \chi^d_3(F)$.
\end{proposition}
\begin{proof}
Let $\phi$ be a $3$-dynamic coloring of $F$ using $\chi_3^d(F)$ colors.
Since $\Delta(F)=3$, we have $\chi_3^d(F)\ge 4$. After relabeling colors, assume
$\phi(u)=0$ and $\phi(v)=3$.
For each diamond in the string replacing $uv$, color its two degree-$3$
vertices with $1$ and $2$, and its two degree-$2$ vertices with $3$ and $0$,
oriented so that the vertex closer to $u$ receives $3$ and the one closer to
$v$ receives $0$.

Then every vertex of the string sees the required number of distinct colors.
Moreover, at $u$ the deleted neighbor $v$ of color $3$ is replaced by a
neighbor of color $3$, and similarly at $v$ the deleted neighbor $u$ of color
$0$ is replaced by a neighbor of color $0$. All other vertices keep the same
neighborhood. Hence $\phi$ extends to a $3$-dynamic coloring of $G$ using no
new colors. Therefore,
$
\chi_3^d(G)\le \chi_3^d(F).
$
\end{proof}
\end{comment}

\begin{proposition}\label{prop:chi3Gleqchi3F}
Let $F$ be a graph with $\Delta(F)=3$.

\begin{enumerate}
    \item If $G$ is obtained from $F$ by replacing an edge $uv$ with a
    string of diamonds, then
    ${\chi_3^d(G)\leq \chi_3^d(F).}$
    \item Let $u$ and $v$ be adjacent vertices of degree two in $F$
    having a common neighbor $w$. If $G$ is obtained from $F$ by adding
    a nonempty string of diamonds joining $u$ and $v$, then
    $
        \chi_3^d(G)\leq \max\{\chi_3^d(F),6\}.
    $
\end{enumerate}
\end{proposition}

\begin{proof}
Let $\phi$ be a $3$-dynamic coloring of $F$.

For the first statement, since $\Delta(F)=3$, at least four colors are
available. Assume
$
\phi(u)=0
\qquad\text{and}\qquad
\phi(v)=3.
$
For each diamond in the string replacing $uv$, color its two degree-$3$
vertices with $1$ and $2$, and its two degree-$2$ vertices with $3$ and
$0$, oriented so that the vertex closer to $u$ receives $3$ and the
vertex closer to $v$ receives $0$.
(Figure \ref{fig:cubic-claw-free-ring-string}b).
Then every vertex of the string sees the required number of distinct
colors. Moreover, at $u$ the deleted neighbor $v$ of color $3$ is
replaced by a neighbor of color $3$, and similarly at $v$ the deleted
neighbor $u$ of color $0$ is replaced by a neighbor of color $0$.
Thus $\phi$ extends to a $3$-dynamic coloring of $G$ using no new
colors, and
$
\chi_3^d(G)\leq\chi_3^d(F).
$

For the second statement, 
%use a palette of $\max\{\chi_3^d(F),6\}$ colors. 
since $u,v,w$ induce a triangle,
we may assume
$
\phi(u)=0,\qquad
\phi(v)=1,\qquad
\phi(w)=2.
$
%Choose three further colors $3,4,5$.
Orient the string from $u$ to $v$. In every diamond, color its two
degree-$3$ vertices with $2$ and $5$, and its two degree-$2$ vertices
with $3$ and $4$, with the vertex closer to $u$ receiving $3$ and the
vertex closer to $v$ receiving $4$. Thus consecutive diamonds are
joined by an edge whose endpoints have colors $4$ and $3$.

The three neighbors of $u$ receive colors $1,2,3$, and the three
neighbors of $v$ receive colors $0,2,4$. Every internal vertex of the
string also has three neighbors of pairwise distinct colors. All other neighborhoods are unchanged.
Hence $\phi$ extends to a $3$-dynamic coloring of $G$ with at most
$\max\{\chi_3^d(F),6\}$ colors.
\end{proof}

Oum characterized bridgeless claw-free cubic graphs in terms of a reduction to cubic multigraphs. To handle graphs with bridges, we extend Proposition~1 of Oum \cite{Oum2011} as follows.

\begin{proposition} \label{prop:oum-extended}
Let $G$ be a connected claw-free cubic graph. Then one of the
following holds:
\begin{enumerate}
    \item $G\cong K_4$;
    \item $G$ is a ring of diamonds;
    \item $G$ is obtained from a connected cubic multigraph $H$,
    possibly containing cutedges and loops, by replacing every
    vertex of $H$ with a triangle and every edge of $H$ with a
    possibly empty string of diamonds. Moreover, every loop of $H$
    is replaced with a nonempty string of diamonds.
\end{enumerate}
\end{proposition}

\begin{proof}
We use the elementary observations appearing in the proof of
Proposition~1 of Oum.
Every vertex of $G$ belongs to a triangle. Moreover, if a vertex
belongs to $3$ triangles, then 
$G\cong K_4$.
Thus, we may assume that every vertex belongs to at most $2$ triangles.
A vertex belonging to two triangles lies in an induced diamond, and
distinct diamonds are vertex-disjoint. Every vertex not contained
in a diamond consequently belongs to exactly one triangle. Hence,
the vertex set of $G$ has a partition $\mathcal{P}$ whose members
induce triangles or diamonds.

Let $B$ be the multigraph obtained by contracting every member of
$\mathcal{P}$ to a single vertex. The edges of $B$ correspond to
edges of $G$ joining distinct members of $\mathcal{P}$.
Note that $B$ is loopless but can have parallel edges.
If $P\in\mathcal{P}$ is a triangle, then the corresponding vertex
of $B$ has degree three. If $P$ is a diamond, then the corresponding
vertex of $B$ has degree two. Therefore, $d_B(v)\in\{2,3\}$
for every $v\in V(B)$.
If every vertex of $B$ has degree two, then $B$ is a cycle, since
$B$ is connected. In this case, $G$ is a ring of diamonds.
Otherwise, suppress all degree-two vertices of $B$. That is, every
maximal path whose internal vertices have degree two is replaced
with one edge. The two endvertices of such a path are allowed to
coincide, in which case the resulting edge is a loop.

Let $H$ be the resulting multigraph. Every vertex of $H$ has degree
three, and hence $H$ is a connected cubic multigraph. (It may contain
parallel edges, bridges and loops).
Every edge $e$ of $H$ corresponds to a maximal path of $B$ whose
internal vertices represent diamonds. If this path has $k$ internal
vertices, then $e$ corresponds to a string of $k$ diamonds.
Reversing the contractions shows that $G$ is obtained from $H$ by replacing its vertices with triangles and its edges with strings of diamonds.
Moreover, if $e$ is a loop of $H$, then its corresponding path in $B$ begins
and ends at the same triangle block. As $B$ is loopless, the path contains at least one diamond block, and the
loop is replaced with a nonempty string of diamonds.
\end{proof}

Finally, we need the following auxiliary result. An \emph{incidence} of a loopless multigraph $H$ is a pair
$(v,e)$, where $v\in V(H)$ and $e\in E(H)$ is incident with $v$.
Two incidences $(u,e)$ and $(v,f)$ are said to be \emph{adjacent} if at
least one of the following conditions holds: $u=v$, $e=f$, or one of $e$ and $f$ is an edge joining $u$ and $v$.
An \emph{incidence list assignment} of $H$ is a function $L$ that
assigns to every incidence $(v,e)$ a nonempty set $L(v,e)$ of
colors. An \emph{incidence $L$-coloring} is a function
$    \varphi$
such that
$
    \varphi(v,e)\in L(v,e)
$
for every incidence $(v,e)$, and adjacent incidences receive distinct
colors.

The multigraph $H$ is \emph{incidence $k$-choosable} if it admits an
incidence $L$-coloring for every incidence list assignment $L$
satisfying
$
    |L(v,e)|\geq k
$
for every incidence $(v,e)$. The \emph{incidence choice number} of $H$,
denoted by $\operatorname{ch}_{i}(H)$, is the least integer $k$ for
which $H$ is incidence $k$-choosable.

We shall use the following theorem of Kang and Park.

\begin{theorem}[Kang and Park~\cite{KangPark2019}]
\label{thm:kang-park}
Every loopless subcubic multigraph $H$ is incidence
$6$-choosable. Equivalently,
$    \operatorname{ch}_{i}(H)\leq 6$.
\end{theorem}

We are ready to show our first theorem.

\begin{theorem} \label{thm:clawfreeleq6}
Let $G$ be a claw-free cubic graph. Then $\chi_3^d(G)\leq 6$.
Moreover, the bound is tight.
\end{theorem}

\begin{proof}
It is enough to consider $G$ connected. If $G\cong K_4$, the result is
immediate. If $G$ is a ring of diamonds, the result follows from
Observation~\ref{obs:ringX3leq5}. Hence, by
Proposition~\ref{prop:oum-extended}, we may assume that $G$ is obtained
from a connected cubic multigraph $H$ by replacing its vertices with
triangles and its edges with possibly empty strings of diamonds.

Remove all loops of $H$, and let $\widehat H$ be the resulting loopless
subcubic multigraph. By Theorem~\ref{thm:kang-park},
$\widehat H$ has an incidence coloring $\phi$ using at most six colors.
Let $\widehat G$ be obtained from $\widehat H$ by replacing every vertex
$x$ with a triangle $T_x$. If $u\in T_x$ corresponds to an edge
$e=xy\in E(\widehat H)$, set
\[
\widehat\phi(u)=\phi(x,e).
\]
If $d_{\widehat H}(x)=3$, then the three colors seen by any vertex of
$T_x$ correspond to three pairwise adjacent incidences of $\widehat H$,
and hence are distinct.
If $d_{\widehat H}(x)=1$, then $T_x$ contains two adjacent vertices of
degree two. Let $w$ be the third vertex of $T_x$ and let $z$ be its
neighbor outside $T_x$. Since $\widehat\phi(w)\neq\widehat\phi(z)$, we
can assign to the two degree-two vertices two distinct colors, both
different from $\widehat\phi(w)$ and $\widehat\phi(z)$. Thus
$\widehat\phi$ extends to a $3$-dynamic coloring of $\widehat G$, and
therefore
$
\chi_3^d(\widehat G)\leq 6.
$
Finally, $G$ is obtained from $\widehat G$ by repeatedly applying the
two operations of Proposition~\ref{prop:chi3Gleqchi3F}: non-loop edges
are replaced by strings of diamonds, whereas each deleted loop is
recovered by adding a string of diamonds between the two corresponding
degree-$2$ vertices of the same triangle. Hence
\[
\chi_3^d(G)\leq\chi_3^d(\widehat G)\leq6.
\]

The bound is tight. Indeed, the triangular prism $C_3\square K_2$ is
cubic and claw-free and has diameter two. In any $3$-dynamic coloring,
two adjacent vertices receive distinct colors, and two vertices with a
common neighbor also receive distinct colors. Hence all six vertices of
$C_3\square K_2$ must receive distinct colors, and therefore
$
\chi_3^d(C_3\square K_2)=6.
$
\end{proof}

We now proceed with the Hamiltonian case.
By Proposition~\ref{prop:oum-extended}, a loop or a
bridge of $H$ would give rise to a bridge in $G$.
Hence we have the next corollary.

\begin{corollary}\label{cor:oum-bridgeless}
Let $G$ be a connected bridgeless claw-free cubic graph which is neither
$K_4$ nor a ring of diamonds. Then $G$ is obtained from a loopless
bridgeless cubic multigraph $H$ by replacing every vertex of $H$ with
a triangle and every edge of $H$ with a possibly empty string of diamonds.
\end{corollary}

Let $X_{12}$ be the graph obtained from the cubic multigraph consisting
of a $4$-cycle with two opposite edges doubled by replacing each
vertex with a triangle. Let $X_{16}^1$ and $X_{16}^2$ be the graphs
obtained from $X_{12}$ by inserting one diamond into, respectively,
an edge belonging to a parallel pair and a simple edge of the
underlying cubic multigraph.
We need the next technical lemma, whose proof is left in the Appendix.

\begin{lemma}\label{lem:small-diamond-insertions}
If $G$ is obtained from $C_3\square K_2$ by inserting one diamond, or
from $X_{12}$ by inserting two diamonds into edges corresponding to the
underlying cubic multigraph, then
$
\chi_3^d(G)\leq5.
$
\end{lemma}

We now proceed to prove the Hamiltonian setting.

\begin{theorem}\label{thm:ham-clawfree}
Let $G$ be a Hamiltonian claw-free cubic graph. 
Then
$
\chi_3^d(G)\leq5.
$
if and only if
$
G\notin\{C_3\square K_2,X_{12},X_{16}^{1},X_{16}^{2}\}.
$
\end{theorem}
\begin{proof}
If $G\cong K_4$, then the result is immediate. If $G$ is a ring of
diamonds, then the result follows from
Observation~\ref{obs:ringX3leq5}. Since $G$ is Hamiltonian, it is
bridgeless. Hence, by
Corollary~\ref{cor:oum-bridgeless}, $G$ is obtained from a loopless
bridgeless cubic multigraph $H$ by replacing every vertex of $H$ with
a triangle and every edge of $H$ with a possibly empty string of diamonds.
Let $t_0,\ldots,t_{r-1}$ be the vertices of $H$, and let $T_i$ be the
triangle replacing $t_i$.

Let $R$ be a Hamiltonian cycle of $G$. Each triangle block has exactly
two edges of $R$ leaving it, and every nonempty string of diamonds is
traversed by $R$ from one end to the other. Hence, after contracting
each triangle block to a vertex and each string of diamonds to an edge,
$R$ yields a Hamiltonian cycle
\[
t_0t_1\ldots t_{r-1}t_0
\]
of $H$.
Let $F$ be the graph obtained from $H$ by replacing each vertex with a
triangle, or equivalently, by replacing each string of diamonds in $G$
with a single edge. We label the vertices of $T_i$ as
$a_i,b_i,c_i$ so that
\[
D=a_0b_0a_1b_1\ldots a_{r-1}b_{r-1}a_0
\]
is a cycle in $F$.

Since $H$ is cubic, the edges of $H$ not belonging to the Hamiltonian
cycle form a perfect matching. Consequently,
$C=\{c_0,\ldots,c_{r-1}\}$ induces a perfect matching $M$ in $F$.
Given a path or cycle $P=u_0u_1\ldots u_k$, we call a coloring $\phi$
of $P$ standard if
$
\phi(u_i)\equiv i \pmod 3.
$

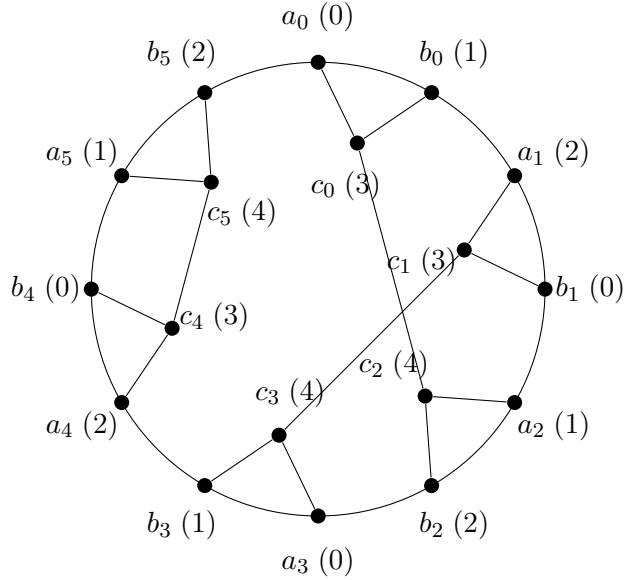
\begin{figure}[H]
    \centering

    % \tikzexternaldisable

\tikzsetnextfilename{fig-cubic-standard}
\begin{tikzpicture}[scale=1, transform shape,
    vertex/.style={circle, fill=black, inner sep=2pt}
]

% Círculo exterior
\draw (0,0) circle (3cm);

% Vértices a_i, b_i
\foreach \i/\name/\col in {
    0/$a_0$/0, 1/$b_0$/1, 2/$a_1$/2, 3/$b_1$/0,
    4/$a_2$/1, 5/$b_2$/2, 6/$a_3$/0, 7/$b_3$/1,
    8/$a_4$/2, 9/$b_4$/0, 10/$a_5$/1, 11/$b_5$/2}
{
    \node[vertex] (v\i) at ({90-30*\i}:3) {};
    \node [shift=({90-30*\i}:0.6)] at (v\i) {\name\;(\col)};
}

% Vértices c_i
\foreach \i/\name/\col in {
    0/$c_0$/3, 1/$c_1$/3, 2/$c_2$/4,
    3/$c_3$/4, 4/$c_4$/3, 5/$c_5$/4}
{
    \node[vertex] (c\i) at ({75-60*\i}:2) {};
    \node [shift=({75-60*\i}:-0.58)] at (c\i) {\name\;(\col)};
}

% Aristas a_i-c_i y b_i-c_i
\foreach \i in {0,...,5}
{
    \pgfmathtruncatemacro{\a}{2*\i}
    \pgfmathtruncatemacro{\b}{2*\i+1}
    \draw (v\a)--(c\i)--(v\b);
}

% Aristas del matching M
\draw (c0)--(c2);
\draw (c1)--(c3);
\draw (c4)--(c5);

\end{tikzpicture}

    \caption{The case $2r\equiv0\pmod3$ in the proof of
    Theorem \ref{thm:ham-clawfree}. In this case $2r=12$.
    In parentheses after each vertex label is shown the color
    assigned to that vertex.}
    \label{fig:2requiv0mod3}
\end{figure}

Whenever $\chi_3^d(F)\leq5$, repeated applications of
Proposition~\ref{prop:chi3Gleqchi3F} yield $\chi_3^d(G)\leq5$.
We now construct a $3$-dynamic $5$-coloring of $F$, except for the
small configurations that will be handled separately using
Lemma~\ref{lem:small-diamond-insertions}.
Suppose first that $2r\equiv0\pmod3$. Precolor $D$ with a standard
coloring $\phi$. Since $2r$ is a multiple of $3$, this coloring is
well defined and proper on $D$. For every $i$ (indices taken modulo
$r$), the vertices $b_{i-1}$ and $b_i$ receive distinct colors, and
so do $a_i$ and $a_{i+1}$.
Since $C=\{c_0,\ldots,c_{r-1}\}$ induces the perfect matching $M$,
extend $\phi$ by assigning colors $3$ and $4$ to the two endpoints of
each edge of $M$. Then $a_i$ sees distinct colors on $b_{i-1}$ and
$b_i$, together with the color of $c_i$, and $b_i$ sees distinct
colors on $a_i$ and $a_{i+1}$, together with the color of $c_i$.
Moreover, $c_i$ sees the two distinct colors of $a_i$ and $b_i$ and
the color of its neighbor in $M$. Hence every vertex has three
neighbors of pairwise distinct colors, and $\phi$ is a $3$-dynamic
$5$-coloring of $F$ (Figure~\ref{fig:2requiv0mod3}).
We may therefore assume from now on that
$2r\not\equiv0\pmod3$.

Suppose that every edge of $M$ joins two consecutive vertices of $C$.
After relabeling if necessary, we may assume that
\[
M=\{c_{2i}c_{2i+1}:0\leq i<r/2\}.
\]
We first consider the case $2r\equiv1\pmod3$. Since $r$ is even, either
$r=2$ or $r\geq8$.
If $r=2$, then $F\cong C_3\square K_2$. Since
$\chi_3^d(C_3\square K_2)=6$, we may assume that at least one diamond occurs in $G$.
Let $F'$ be obtained from $F$ by retaining exactly one such diamond.
By Lemma~\ref{lem:small-diamond-insertions},
$\chi_3^d(F')\leq5$. Since $G$ is obtained from $F'$ by inserting the
remaining diamonds, repeated applications of
Proposition~\ref{prop:chi3Gleqchi3F} yield $\chi_3^d(G)\leq5$.

Hence, we may assume that $r\geq8$. We precolor
$
a_0b_0a_1b_1\ldots a_{r-3}b_{r-3}
$
with a standard coloring $\phi$. We then set
$
\phi(c_0)=\phi(c_{r-3})=\phi(c_{r-1})=3,
\qquad
\phi(c_{r-2})=0,
$
and
$
\phi(a_{r-2})=\phi(b_{r-1})=4,\qquad
\phi(b_{r-2})=1,\qquad
\phi(a_{r-1})=2.
$
We color the remaining vertices of $C$ with colors $3$ and $4$ so that
the endpoints of every edge of $M$ receive distinct colors.
A direct check shows that the resulting coloring is proper and that the
three neighbors of every vertex receive pairwise distinct colors.
Hence $\phi$ is a $3$-dynamic $5$-coloring of $F$
(Figure~\ref{fig:cubic-consecutive}a).

We now consider the case $2r\equiv2\pmod3$. Since $r$ is even,
$r\equiv4\pmod6$. We first assume that $r\geq10$.
We precolor
\[
a_0b_0a_1b_1\ldots a_{r-5}b_{r-5}
\]
with a standard coloring $\phi$. We then set
$
\phi(c_0)=\phi(c_{r-1})=\phi(b_{r-3})
=\phi(a_{r-4})=3,
$
$
\phi(a_{r-2})=\phi(b_{r-1})
=\phi(c_{r-4})=\phi(c_{r-5})=4,
$
$
\phi(c_{r-2})=\phi(c_{r-3})=0,
$ and
$
\phi(a_{r-1})=\phi(b_{r-4})=1,
\qquad
\phi(b_{r-2})=\phi(a_{r-3})=2.
$
We color the remaining vertices of $C$ with colors $3$ and $4$ so that
the endpoints of every edge of $M$ receive distinct colors.
A direct check shows that the resulting coloring is proper and that the
three neighbors of every vertex receive pairwise distinct colors.
Hence $\phi$ is a $3$-dynamic $5$-coloring of $F$
(Figure~\ref{fig:cubic-consecutive}b).
Finally, If $r=4$, then
$
M=\{c_0c_1,c_2c_3\},
$
and hence $F\cong X_{12}$. If no diamond occurs, then
$G\cong X_{12}$; if exactly one diamond occurs, then
$G\cong X_{16}^1$ or $G\cong X_{16}^2$. It can be verified that $\chi_3^d(G)=6$ for any such $G$. Hence, assume that at least two diamonds occur in $G$.
Let $F'$ be obtained from $X_{12}$ by retaining exactly two of these
diamonds. By Lemma~\ref{lem:small-diamond-insertions},
$\chi_3^d(F')\leq5$. Since $G$ is obtained from $F'$ by inserting the
remaining diamonds, repeated applications of
Proposition~\ref{prop:chi3Gleqchi3F} yield $\chi_3^d(G)\leq5$.

\begin{figure}[H]
    \centering

    \begin{subfigure}{0.48\textwidth}
        \centering
        % \tikzexternaldisable

\tikzsetnextfilename{fig-cubic-consecutive-1}
\begin{tikzpicture}[scale=1,
    vertex/.style={circle, fill=black, inner sep=2pt}
]

% Círculo exterior
\draw (0,0) circle (3cm);

% Vértices a_i, b_i
\foreach \i/\name/\col in {
    0/$a_0$/0, 1/$b_0$/1, 2/$a_1$/2, 3/$b_1$/0,
    4/$a_2$/1, 5/$b_2$/2, 6/$a_3$/0, 7/$b_3$/1,
    8/$a_4$/2, 9/$b_4$/0, 10/$a_5$/1, 11/$b_5$/2,
    12/$a_6$/4, 13/$b_6$/1, 14/$a_7$/2, 15/$b_7$/4}
{
    \node[vertex] (v\i) at ({90-22.5*\i}:3) {};
    \node [shift=({90-22.5*\i}:0.65)] at (v\i) {\name\;(\col)};
}

% Vértices c_i, centrados entre a_i y b_i
\foreach \i/\name/\col in {
    0/$c_0$/3, 1/$c_1$/4, 2/$c_2$/3,
    3/$c_3$/4, 4/$c_4$/4, 5/$c_5$/3,
    6/$c_6$/0, 7/$c_7$/3}
{
    \node[vertex] (c\i) at ({78.75-45*\i}:2.4) {};
    \node [shift=({78.75-45*\i}:-0.58)] at (c\i) {\name\;(\col)};
}

% Aristas a_i-c_i y b_i-c_i
\foreach \i in {0,...,7}
{
    \pgfmathtruncatemacro{\a}{2*\i}
    \pgfmathtruncatemacro{\b}{2*\i+1}
    \draw (v\a)--(c\i)--(v\b);
}

% Aristas del matching M
\draw (c0)--(c1);
\draw (c2)--(c3);
\draw (c4)--(c5);
\draw (c6)--(c7);

\end{tikzpicture}
        \caption{$ $}
    \end{subfigure}
    \hfill
    \begin{subfigure}{0.48\textwidth}
        \centering
        \tikzsetnextfilename{fig-cubic-consecutive-2}
\begin{tikzpicture}[scale=1,
    vertex/.style={circle, fill=black, inner sep=2pt}
]

% Círculo exterior
\draw (0,0) circle (3cm);

% Vértices a_i, b_i
\foreach \i/\name/\col in {
    0/$a_0$/0,  1/$b_0$/1,
    2/$a_1$/2,  3/$b_1$/0,
    4/$a_2$/1,  5/$b_2$/2,
    6/$a_3$/0,  7/$b_3$/1,
    8/$a_4$/2,  9/$b_4$/0,
    10/$a_5$/1, 11/$b_5$/2,
    12/$a_6$/3, 13/$b_6$/1,
    14/$a_7$/2, 15/$b_7$/3,
    16/$a_8$/4, 17/$b_8$/2,
    18/$a_9$/1, 19/$b_9$/4}
{
    \node[vertex] (v\i) at ({90-18*\i}:3) {};
    \node [shift=({90-18*\i}:0.65)] at (v\i) {\name\;(\col)};
}

% Vértices c_i, centrados entre a_i y b_i
\foreach \i/\name/\col in {
    0/$c_0$/3, 1/$c_1$/4,
    2/$c_2$/3, 3/$c_3$/4,
    4/$c_4$/3, 5/$c_5$/4,
    6/$c_6$/4, 7/$c_7$/0,
    8/$c_8$/0, 9/$c_9$/3}
{
    \node[vertex] (c\i) at ({81-36*\i}:2.4) {};
    \node [shift=({81-36*\i}:-0.58)] at (c\i) {\name\;(\col)};
}

% Aristas a_i-c_i y b_i-c_i
\foreach \i in {0,...,9}
{
    \pgfmathtruncatemacro{\a}{2*\i}
    \pgfmathtruncatemacro{\b}{2*\i+1}
    \draw (v\a)--(c\i)--(v\b);
}

% Aristas del matching M
\draw (c0)--(c1);
\draw (c2)--(c3);
\draw (c4)--(c5);
\draw (c6)--(c7);
\draw (c8)--(c9);

\end{tikzpicture}
        %\caption{$2r\equiv2\pmod3$, $2r=20$.}
        \caption{$ $}
    \end{subfigure}

  \caption{The case in which every edge of $M$ joins two consecutive
vertices of $C$ in the proof of Theorem~\ref{thm:ham-clawfree}.
In parentheses after each vertex label is the color assigned to that
vertex. In (a), $2r\equiv1\pmod3$ and $2r=16$; in (b),
$2r\equiv2\pmod3$ and $2r=20$.}
    \label{fig:cubic-consecutive}
\end{figure}
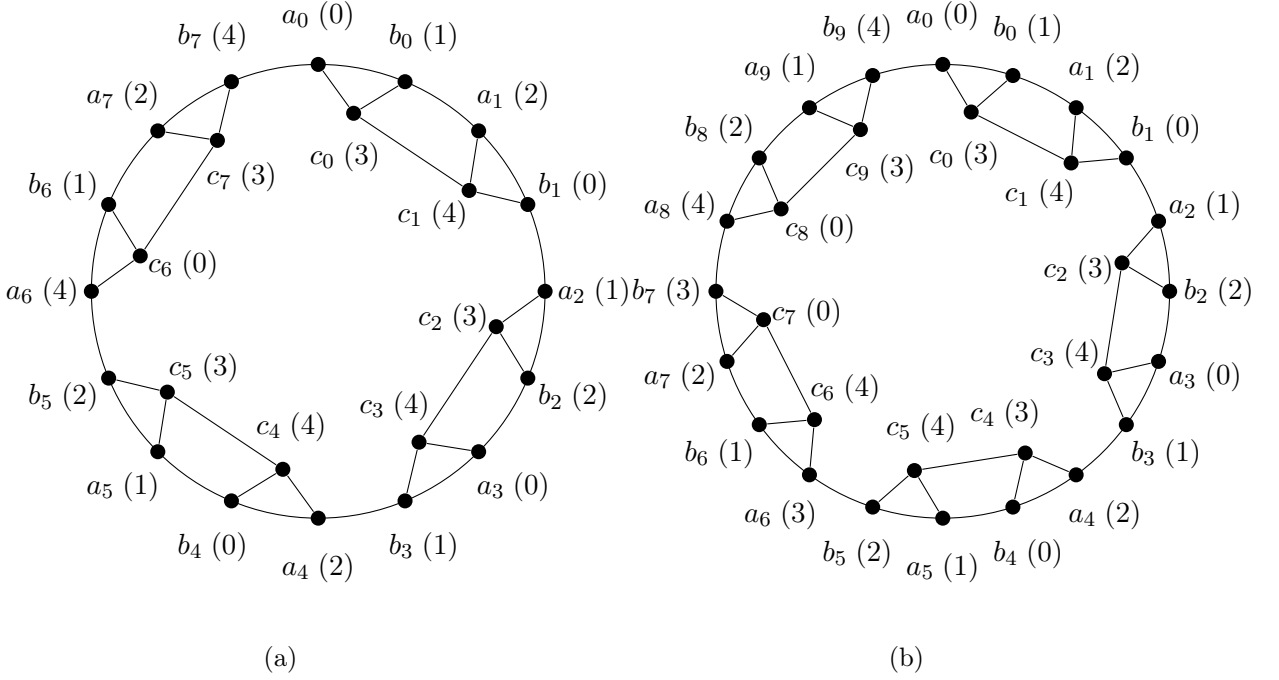

We may therefore assume that $M$ contains an edge joining two
nonconsecutive vertices of $C$.
Among the edges $c_ic_k\in M$ joining nonconsecutive vertices of $C$,
choose one for which
$
\min\{|i-k|,\,r-|i-k|\}
$
is maximum. After relabeling the indices and, if necessary, reversing
their cyclic order, we may assume that this edge is $c_0c_j$, where
$
2\leq j\leq r/2.
$
Let
$
P=a_1b_1a_2b_2\ldots a_{j-1}b_{j-1}
$
and
$
Q=a_{j+1}b_{j+1}a_{j+2}b_{j+2}\ldots
a_{r-1}b_{r-1}.
$
Since
$
|V(P)|+|V(Q)|=2r-4
$
and $2r\not\equiv0\pmod3$, we have
$
|V(P)|+|V(Q)|\not\equiv2\pmod3.
$
Moreover, since $2\leq j\leq r-2$, both $P$ and $Q$ are nonempty.

Let $\overline{Q}$ denote the path $Q$ taken in the reverse order.
We precolor $P$ and $\overline{Q}$ with standard colorings $\phi$.
Thus, $\phi(a_1)=\phi(b_{r-1})=0$. We also set
$\phi(a_0)=\phi(c_1)=3$ and
$\phi(b_0)=\phi(c_{r-1})=4$.
We claim that $c_{j-1}$ and $c_{j+1}$ can be assigned distinct colors
from $\{3,4\}$ while preserving properness on the already colored edges
of $M$. If $j\geq3$, this follows directly from the fact that $M$ is a
matching. Suppose then that $j=2$. Since $c_1$ already has color $3$, we only
need to show that $c_3$ can receive color $4$. The only possible
obstruction is $c_3c_{r-1}\in M$. If $r\geq8$, this contradicts the
maximal choice of $c_0c_2$, since $c_3c_{r-1}$ has larger cyclic
distance. If $r=4$, then $c_3=c_{r-1}$ already has color $4$.

Thus, $c_{j-1}$ and $c_{j+1}$ can be assigned distinct colors from
$\{3,4\}$. By symmetry, we may assume that
$\phi(c_{j-1})=3$ and $\phi(c_{j+1})=4$. We then set
$\phi(a_j)=4$ and $\phi(b_j)=3$.
We color every remaining vertex of
$C\setminus\{c_0,c_j\}$ with colors $3$ and $4$ so that the endpoints
of every edge of $M$ other than $c_0c_j$ receive distinct colors.
At this point, every vertex outside $T_0\cup T_j$ satisfies the
$3$-dynamic condition. It remains to color $c_0$ and $c_j$.

First suppose that one of $|V(P)|$ and $|V(Q)|$ is congruent to $0$
modulo $3$. By symmetry, assume that $|V(P)|\equiv0\pmod3$.
Since $b_{j-1}$ is the last vertex of $P$, we have
$\phi(b_{j-1})=2$. Moreover, since
$|V(P)|+|V(Q)|\not\equiv2\pmod3$, we have
$|V(Q)|\not\equiv2\pmod3$. As $\overline{Q}$ is standard-colored,
it follows that $\phi(a_{j+1})\in\{0,2\}$.
We set $\phi(c_0)=2$ and $\phi(c_j)=1$. Since
$\phi(a_0)=3$, $\phi(b_0)=4$, and
$\phi(a_1)=\phi(b_{r-1})=0$, the coloring is proper and the
$3$-dynamic condition holds at the vertices of $T_0$.
Furthermore,
\[
\phi(b_{j-1})=2,\qquad
\phi(a_j)=4,\qquad
\phi(b_j)=3,\qquad
\phi(a_{j+1})\in\{0,2\}.
\]
Hence the coloring is proper and the $3$-dynamic condition also holds
at the vertices of $T_j$. Therefore, $\phi$ is a $3$-dynamic
$5$-coloring of $F$ (Figure~\ref{fig:cubic-notconsecutive}a).

\begin{figure}[H]
    \centering

    \begin{subfigure}{0.48\textwidth}
        \centering
        \tikzsetnextfilename{fig-cubic-P0mod3}
\begin{tikzpicture}[scale=1,
    vertex/.style={circle, fill=black, inner sep=2pt}
]

% Círculo exterior
\draw (0,0) circle (3cm);

% Vértices a_i, b_i
\foreach \i/\name/\col in {
    0/$a_0$/3,  1/$b_0$/4,
    2/$a_1$/0,  3/$b_1$/1,
    %4/$a_2$/1,  5/$b_2$/2,
    %6/$a_3$/0,  7/$b_3$/1,
    8/$a_{j-1}$/1,  9/$b_{j-1}$/2,
    10/$a_j$/4, 11/$b_j$/3,
    12/$a_{j+1}$/{0,2}, 13/$b_{j+1}$/{2,1},
    %14/$a_7$/2, 15/$b_7$/3,
    %16/$a_8$/4, 17/$b_8$/2,
    18/$a_{r-1}$/1, 19/$b_{r-1}$/0}
{
    \node[vertex] (v\i) at ({90-18*\i}:3) {};
    \node [shift=({90-18*\i}:0.7)] at (v\i) {\name\;(\col)};
}

% Vértices c_i, centrados entre a_i y b_i
\foreach \i/\name/\col in {
    0/$c_0$/2, 1/$c_1$/3,
    %2/$c_2$/3, 3/$c_3$/4,
    4/$c_{j-1}$/3, 5/$c_j$/1,
    6/$c_{j+1}$/4,
    %7/$c_7$/0, 8/$c_8$/0,
    9/$c_{r-1}$/4}
{
    \node[vertex] (c\i) at ({81-36*\i}:2.5) {};
    \node [shift=({81-36*\i}:-0.48)] at (c\i) {\name\;(\col)};
}

% Aristas a_i-c_i y b_i-c_i
% \foreach \i in {0,1,2,6,7,8}
\foreach \i in {0,1,4,5,6,9}
{
    \pgfmathtruncatemacro{\a}{2*\i}
    \pgfmathtruncatemacro{\b}{2*\i+1}
    \draw (v\a)--(c\i)--(v\b);
}

% Aristas del matching M
\draw (c0)--(c5);
%\draw (c2)--(c3);
%\draw (c4)--(c5);
%\draw (c6)--(c7);
%\draw (c8)--(c9);

\end{tikzpicture}
        \caption{$ $}
        \label{fig:2requiv1mod3}
    \end{subfigure}
    \hfill
    \begin{subfigure}{0.48\textwidth}
        \centering
        \tikzsetnextfilename{fig-cubic-P2mod3}
\begin{tikzpicture}[scale=1, 
    vertex/.style={circle, fill=black, inner sep=2pt}
]

% Círculo exterior
\draw (0,0) circle (3cm);

% Vértices a_i, b_i
\foreach \i/\name/\col in {
    0/$a_0$/3,  1/$b_0$/4,
    2/$a_1$/0,  3/$b_1$/1,
    %4/$a_2$/1,  5/$b_2$/2,
    %6/$a_3$/0,  7/$b_3$/1,
    8/$a_{j-1}$/0,  9/$b_{j-1}$/1,
    10/$a_j$/4, 11/$b_j$/3,
    12/$a_{j+1}$/{1,0}, 13/$b_{j+1}$/{0,2},
    %14/$a_7$/2, 15/$b_7$/3,
    %16/$a_8$/4, 17/$b_8$/2,
    18/$a_{r-1}$/1, 19/$b_{r-1}$/0}
{
    \node[vertex] (v\i) at ({90-18*\i}:3) {};
    \node [shift=({90-18*\i}:0.7)] at (v\i) {\name\;(\col)};
}

% Vértices c_i, centrados entre a_i y b_i
\foreach \i/\name/\col in {
    0/$c_0$/1, 1/$c_1$/3,
    %2/$c_2$/3, 3/$c_3$/4,
    4/$c_{j-1}$/3, 5/$c_j$/2,
    6/$c_{j+1}$/4,
    %7/$c_7$/0, 8/$c_8$/0,
    9/$c_{r-1}$/4}
{
    % \node[vertex] (c\i) at ({81-36*\i}:2.2) {};
    \node[vertex] (c\i) at ({81-36*\i}:2.5) {};
    \node [shift=({81-36*\i}:-0.48)] at (c\i) {\name\;(\col)};
}

% Aristas a_i-c_i y b_i-c_i
\foreach \i in {0,1,4,5,6,9}
% \foreach \i in {0,1,2,6,7,8}
{
    \pgfmathtruncatemacro{\a}{2*\i}
    \pgfmathtruncatemacro{\b}{2*\i+1}
    \draw (v\a)--(c\i)--(v\b);
}

% Aristas del matching M
\draw (c0)--(c5);
%\draw (c2)--(c3);
%\draw (c4)--(c5);
%\draw (c6)--(c7);
%\draw (c8)--(c9);

\end{tikzpicture}
        %\caption{$2r\equiv2\pmod3$, $2r=20$.}
        \caption{$ $}
        \label{fig:cubic-P0mod3}
    \end{subfigure}

    \caption{Situation in the proof
    of Theorem \ref{thm:ham-clawfree}. In this case, there exist an edge $c_0c_j$ with $j>0$. In parenthesis after each
    vertex label is shown the color assigned to such vertex. In $(a)$ we depict the case $|V(P)|\equiv 0\pmod3$.
    In $(b)$ we depict the case $|V(P)|\equiv 2\pmod3$.}
    \label{fig:cubic-notconsecutive}
\end{figure}
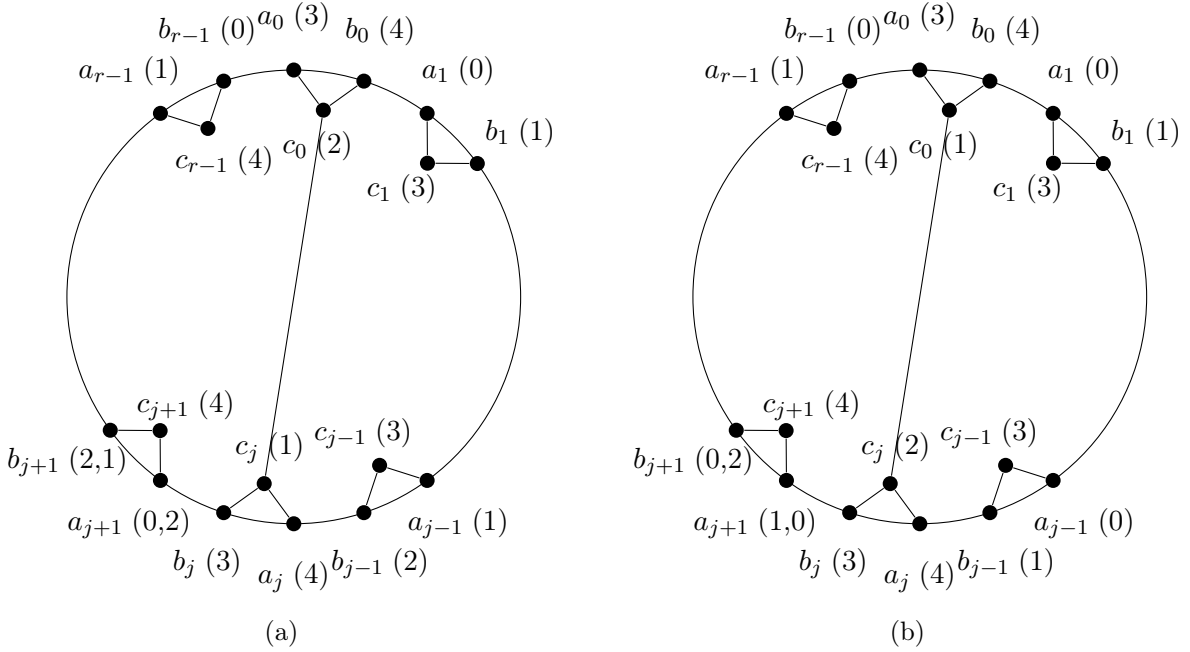

We may therefore assume that neither $|V(P)|$ nor $|V(Q)|$ is
congruent to $0$ modulo $3$. Since
$|V(P)|+|V(Q)|\not\equiv2\pmod3$, at least one of them is congruent
to $2$ modulo $3$. By symmetry, assume that
$|V(P)|\equiv2\pmod3$.
Then $\phi(b_{j-1})=1$. Moreover, $|V(Q)|\not\equiv0\pmod3$, and,
since $\overline{Q}$ is standard-colored,
$\phi(a_{j+1})\in\{0,1\}$.
We set $\phi(c_0)=1$ and $\phi(c_j)=2$. Again, since
$\phi(a_0)=3$, $\phi(b_0)=4$, and
$\phi(a_1)=\phi(b_{r-1})=0$, the coloring is proper and the
$3$-dynamic condition holds at the vertices of $T_0$.
Moreover,
\[
\phi(b_{j-1})=1,\qquad
\phi(a_j)=4,\qquad
\phi(b_j)=3,\qquad
\phi(a_{j+1})\in\{0,1\}.
\]
Thus the coloring is proper and the $3$-dynamic condition also holds
at the vertices of $T_j$. Therefore, $\phi$ is a $3$-dynamic
$5$-coloring of $F$ (Figure~\ref{fig:cubic-notconsecutive}b).
\end{proof}

\section{$2$-dynamic coloring of $C_p(1,3)$}

We use color words cyclically on the vertices
$0,1,\ldots,p-1$; as usual, $W^q$ denotes the concatenation of
$q$ copies of a word $W$.

\begin{theorem}\label{thm:circulant-r2}
Let $p\ge6$. Then
\[
\chi_2^d(C_p(1,3))=
\begin{cases}
3, & p\equiv0\pmod9,\\
4, & \text{otherwise}.
\end{cases}
\]
\end{theorem}

\begin{proof}
If $9\mid p$, the word
$
(012120201)^{p/9}
$
defines a $2$-dynamic $3$-coloring: each vertex has a color different
from its neighbors at distances $1$ and $3$, and its neighborhood
contains both other colors. Hence,
$\chi_2^d(C_p(1,3))=3$ by
Equation~\ref{eq:absolute-bounds}.

We next give a $2$-dynamic $4$-coloring for every $p\ge6$.
Let $A=0123$ and write $p=4q+r$, $0\le r\le3$. Use the following
cyclic color words:
\[
\begin{array}{c|cccc}
r&0&1&2&3\\ \hline
\phi&A^q&A^{q-1}03231&A^q01&A^q201
\end{array}
\]
A direct verification shows that $\phi$ is a $2$-dynamic coloring.
Thus $\chi_2^d(C_p(1,3))\le4$ for every $p\ge6$.

It remains to prove that a $2$-dynamic $3$-coloring can exist only if
$9\mid p$. We may assume $p\ge7$, since
$C_6(1,3)\cong K_{3,3}$ has no $2$-dynamic $3$-coloring.
Let $\phi$ be a $2$-dynamic $3$-coloring and define
\[
s_i=
\begin{cases}
1,&\phi(i-1)\neq\phi(i+1),\\
0,&\phi(i-1)=\phi(i+1).
\end{cases}
\]
We first observe that $s_i=s_{i+1}=1$ is impossible. Indeed, in this
case both
$
\phi(i-1),\phi(i),\phi(i+1)
$
and
$
\phi(i),\phi(i+1),\phi(i+2)
$
contain the three colors. Hence $\phi(i-1)=\phi(i+2)$, contradicting
the fact that $i-1$ and $i+2$ are adjacent.

Moreover,
$
s_{i-2}=s_i=s_{i+2}=0
$
is impossible, since it would imply
$
\phi(i-3)=\phi(i-1)=\phi(i+1)=\phi(i+3),
$
contrary to the $2$-dynamic condition at $i$.

Since not all $s_i$ can be zero, choose $i$ with $s_i=1$.
Then $s_{i-1}=s_{i+1}=0$, and the previous observation forces
$s_{i-3}=s_{i+3}=1$. By the first observation,
$s_{i-2}=s_{i+2}=0$.
Reconstructing $\phi$ yields $\phi(i-4)=\phi(i+5)$.
Repeating the same argument gives
\[
\phi(i-4)=\phi(i-4+9k)
\]
for every $k$, and hence $p\equiv0\pmod9$.
\end{proof}
\section{$3$-dynamic coloring of $C_p(1,3)$}

\begin{theorem}\label{thm:circulant-r3}
Let $p\ge12$. Then
\[
\chi_3^d(C_p(1,3))=
\begin{cases}
4, & p\equiv0\pmod{12},\\
5, & \text{otherwise}.
\end{cases}
\]
\end{theorem}

\begin{proof}
Suppose first that $12\mid p$. The cyclic color word
$
(010232101323)^{p/12}
$
defines a proper $3$-dynamic $4$-coloring. Hence
$\chi_3^d(C_p(1,3))=4$ by
Equation~\ref{eq:absolute-bounds}.

We next give a $3$-dynamic $5$-coloring for every $p\ge12$.
Let $A=01234$ and write $p=5q+r$, where $0\le r\le4$.
Consider the following cyclic color words:
\[
\begin{array}{c|ccccc}
r&0&1&2&3&4\\ \hline
\phi&
A^q&
A^{q-1}102431&
A^q01&
A^{q-1}01314203&
A^q1203
\end{array}
\]
A direct verification shows that each word defines a proper
$3$-dynamic $5$-coloring. Thus
$
\chi_3^d(C_p(1,3))\le5
$
for every $p\ge12$.

It remains to show that a $3$-dynamic $4$-coloring can exist only if
$12\mid p$. Let $\phi$ be such a coloring. Since $\phi$ is proper,
the neighbors of every vertex use exactly the three colors different
from the color of the vertex. Thus exactly two of its four neighbors
receive the same color.

Define
\[
s_i=
\begin{cases}
1,&\phi(i-1)=\phi(i+1),\\
0,&\phi(i-1)\neq\phi(i+1).
\end{cases}
\]
We first observe that the repeated color in $N(i)$ can only occur on
one of the pairs
$
\{i-3,i-1\},\qquad
\{i-1,i+1\},\qquad
\{i+1,i+3\}.
$
Indeed, if $\phi(i-3)=\phi(i+1)$, then properness forces
$\phi(i+2)=\phi(i)$, and the $3$-dynamic condition at $i-1$ forces
one of $i-4,i-2$ to receive the color of $i-3$, contradicting
properness. The case $\phi(i-1)=\phi(i+3)$ is symmetric. If
$\phi(i-3)=\phi(i+3)$, then properness forces
$
\phi(i-2)=\phi(i+2)=\phi(i),
$
contradicting the $3$-dynamic condition at $i-1$.

Consequently, for every $i$, exactly one of
\[
s_{i-2},\ s_i,\ s_{i+2}
\]
is equal to $1$. Moreover, two consecutive $s_i$ cannot both be $1$.
Indeed, if $s_i=s_{i+1}=1$, a direct reconstruction using the
$3$-dynamic condition around $i$ and $i+1$ forces
$\phi(i+4)=\phi(i+5)$, contradicting properness.

Choose $i$ with $s_i=1$. Then by the previous
observations, $s_{i-6}=s_{i+6}=s_{i-3}=s_{i+3}=1$ and every other $s_k$ between $i-6$ and $i+6$ equals 0.
This implies $3\mid p$.
Reconstructing $\phi$ yields
$\phi(i-6)=\phi(i+6)$.
Repeating the same argument gives
\[
\phi(i-6)=\phi(i-6+12k) \neq \phi(i)
\]
for every $k$, and hence, as $p\equiv0\pmod {12}$.
\end{proof}
\section{$4$-dynamic coloring of $C_p(1,3)$}

\begin{theorem}\label{thm:circulant-r4}
Let $p\ge12$. Then
\[
\chi_4^d(C_p(1,3))=
\begin{cases}
5, & p\equiv0\pmod5,\\
8, & p=16,\\
7, & p\in\{13,14,19,21,26,31\},\\
6, & \text{otherwise}.
\end{cases}
\]
\end{theorem}

\begin{proof}
Since $C_p(1,3)$ is $4$-regular, a proper coloring is $4$-dynamic
if and only if vertices at distance at most two receive different
colors. Hence
\[
\chi_4^d(C_p(1,3))=\chi(H_p),
\]
where
$H_p=C_p(1,2,3,4,6)$.
For every $i$, the set
$
\{i,i+1,i+2,i+3,i+4\}
$
induces a $K_5$ in $H_p$. Thus $\chi(H_p)\ge5$. Moreover, in any
$5$-coloring, two consecutive such cliques force
$
\phi(i+5)=\phi(i)
$
for every $i$. Therefore a $5$-coloring can exist only if $5\mid p$.
Consequently,
$
\chi(H_p)\ge6
$
whenever $5\nmid p$.

We need stronger lower bounds only for a few small values of $p$.
Let $S$ be an independent set of $H_p$, and let
$g_1,\ldots,g_t$ be the cyclic distances between consecutive vertices of
$S$. Since the distances $1,2,3,4,6$ are forbidden, every $g_j$ is
either $5$ or at least $7$. Hence
\[
\alpha(H_p)\le\left\lfloor\frac p5\right\rfloor.
\]
Furthermore, if $p=5q+1$, then $\alpha(H_p)\le q-1$, since $q$
distances would require
$
1=\sum_{j=1}^q(g_j-5),
$
where every summand is either $0$ or at least $2$.
Using
$
\chi(H_p)\ge
\left\lceil\frac{p}{\alpha(H_p)}\right\rceil,
$
we obtain
$
\chi(H_{16})\ge8
$
and
$
\chi(H_p)\ge7
\qquad
\text{for }p\in\{13,14,19,21,26,31\}.
$

It remains to give matching colorings. Let
\[
A=01234,\qquad
B=012340521435,\qquad
C=012340153421504235.
\]
If $p=5q$, then $A^q$ is a proper $5$-coloring of $H_p$.

For the nonexceptional values with $5\nmid p$, write $p=5q+r$.
The following cyclic color words use six colors:
\[
\begin{array}{c|c|c}
r & \text{condition} & \phi\\ \hline
1 & q\ge7 & A^{q-7}C^2\\
2 & q\ge2 & A^{q-2}B\\
3 & q\ge3 & A^{q-3}C\\
4 & q\ge4 & A^{q-4}B^2.
\end{array}
\]
A direct verification shows that each word is a proper coloring of
$H_p$.

Finally, let
\[
D=0123456,\qquad E=01234567.
\]
The exceptional cases are covered by
\[
\begin{array}{c|c}
p&\phi\\ \hline
13&D201534\\
14&D^2\\
16&E^2\\
19&D^2A\\
21&D^3\\
26&D^3A\\
31&D^4\,234.
\end{array}
\]
Again, a direct verification shows that these are proper colorings of
$H_p$. They use seven colors, except for $p=16$, where eight colors
are used. Together with the lower bounds above, the result follows.
\end{proof}

\section{Concluding remarks}

In this paper, we investigated \(r\)-dynamic colorings of regular graphs for claw-free cubic graphs and circulant graphs. 
For cubic graphs, 3-dynamic coloring coincides with coloring the square of the graph. This connection allows the problem to be viewed from the perspective of square coloring while exploiting the strong structural properties of claw-free cubic graphs. We proved that every claw-free cubic graph admits a 3-dynamic coloring with at most six colors, and we identified the extremal cases showing that this bound is best possible. For Hamiltonian claw-free cubic graphs, the additional structure permits a stronger bound of five colors, apart from the exceptional graphs \(C_3\square K_2\), \(X_{12}\), \(X_{16}^{1}\), and \(X_{16}^{2}\).

The circulant case illustrates a substantially different phenomenon. Here the parameter \(r\) is smaller than the degree, and consequently \(r\)-dynamic coloring no longer reduces to coloring a graph square. For the circulants considered in this work, this distinction allows considerably more economical colorings and leads to exact results depending on the arithmetic structure of the underlying cycle.

Several questions remain open. A natural starting point is the general situation for \(r\)-regular graphs. Since an \(r\)-dynamic coloring of an \(r\)-regular graph is precisely a proper coloring of its square, one always has
$
r+1\leq \chi_r^d(G)=\chi(G^2)\leq r^2+1.
$
The upper bound follows immediately from \(\Delta(G^2)\leq r^2\), and it is best possible in general, as shown by the classical Moore graph examples. Thus, even in the regular setting, the possible range for the \(r\)-dynamic chromatic number is quite large.

For cubic graphs, this general estimate gives
$
4\leq \chi_3^d(G)\leq 10,
$
and the upper bound cannot be improved over the class of all cubic graphs, since the Petersen graph satisfies \(\chi_3^d(G)=10\). From this perspective, the bounds obtained here show that forbidding an induced claw imposes a substantial restriction on the square coloring problem: every claw-free cubic graph in our setting requires at most six colors, and under the additional Hamiltonicity assumption this can be reduced further to five, apart from the exceptional configurations identified above.
It would therefore be interesting to determine which structural conditions on cubic or, more generally, regular graphs force the \(r\)-dynamic chromatic number to lie substantially below the general bound \(r^2+1\). 
\newpage
\section{Appendix: Proof of Lemma \ref{lem:small-diamond-insertions}.}

%\begin{lemma}\label{lem:small-diamond-insertions}
%If $G$ is obtained from $C_3\square K_2$ by inserting one diamond, or from $X_{12}$ by inserting two diamonds into edges corresponding to the underlying cubic multigraph, then $ \chi^d_3(G)\leq5.$
%\end{lemma}

\begin{proof}
For the first part, 
%the automorphisms of the triangular prism act transitively on the three edges corresponding to the edges of the underlying cubic multigraph. Thus,
let the two triangles of the prism be
$a_0b_0c_0$ and $a_1b_1c_1$, and suppose without loss of generality that $c_0c_1$ is replaced
by a diamond. Let $p$ and $q$ be the degree-$2$ vertices of the diamond,
where $p$ is adjacent to $c_0$ and $q$ is adjacent to $c_1$, and let
$x$ and $y$ be its other two vertices. Define
$
\phi(c_0)=\phi(c_1)=0,\quad
\phi(a_0)=\phi(q)=1,\quad
\phi(b_0)=\phi(x)=2,
$
and
$
\phi(a_1)=\phi(p)=3,\qquad
\phi(b_1)=\phi(y)=4.
$
A direct check shows that $\phi$ is proper and that the three neighbors
of every vertex receive pairwise distinct colors. Hence $\phi$ is a
$3$-dynamic $5$-coloring.

For the second part, let
the four triangles of $X_{12}$ be 
$\{a_ib_ic_i: i\in \{0,1,2,3\}\}$ and
$
e_0=b_0a_1,\quad e_1=b_1a_2,\quad
e_2=b_2a_3,\quad e_3=b_3a_0,\quad
e_4=c_0c_1,\quad e_5=c_2c_3
$
be the their edges outside triangles. Up to automorphisms of $X_{12}$, the
possible placements of two diamonds are
$
(e_0,e_0),\quad (e_0,e_1),\quad (e_0,e_2),\quad
(e_0,e_4),\quad (e_1,e_1),\quad (e_1,e_3).
$
%Indeed, these correspond, respectively, to inserting both diamonds into the same edge belonging to a parallel pair of the underlying multigraph; into one such edge and one simple edge; into edges belonging to different parallel pairs; into the two edges of the same parallel pair; twice into the same simple edge; and into the two distinct simple edges.

%Figure~\ref{fig:X12-two-diamonds} gives a $5$-coloring for each of these six configurations. In every case the coloring is proper and the three neighbors of each vertex receive pairwise distinct colors. Therefore, each coloring is $3$-dynamic, and every graph obtained from $X_{12}$ by inserting two diamonds admits a $3$-dynamic $5$-coloring.
There are six configurations, shown in
Figure~\ref{fig:X12-two-diamonds}; the displayed colorings verify the claim.
\end{proof}

% En el preámbulo:
% \usepackage{tikz,subcaption}

% ================================================================
% Helpers
% ================================================================

\def\Xcfdiamond#1#2#3#4#5{%
  \node[Xcf vertex,
    label={[Xcf label]#4:{$#5$}}]
  (#1) at (#2,#3) {};
}

\tikzset{
  Xcf vertex/.style={
    circle,
    fill=black,
    inner sep=1.35pt,
    outer sep=0pt
  },
  Xcf label/.style={
    font=\scriptsize,
    inner sep=1pt,
    fill=white,
    text=black
  },
  Xcf picture/.style={
    x=.78cm,
    y=.78cm,
    line width=.45pt,
    label distance=2pt
  }
}

% Coordinates and edges common to all six panels.
\def\Xcfbase{%
  \path[use as bounding box]
    (-3.5,-2.65) rectangle (3.5,2.65);

  \foreach \n/\x/\y in {
    a0/-2.4/1.65,
    b0/-1.6/1.65,
    c0/-2/.65,
    a1/1.6/1.65,
    b1/2.4/1.65,
    c1/2/.65,
    a2/2.4/-1.65,
    b2/1.6/-1.65,
    c2/2/-.65,
    a3/-1.6/-1.65,
    b3/-2.4/-1.65,
    c3/-2/-.65
  }{
    \coordinate (\n) at (\x,\y);
  }

  \foreach \i in {0,...,3}{
    \draw (a\i)--(b\i)--(c\i)--cycle;
  }
}

% Draw each base vertex with only its color.
\def\Xcfcolors#1{%
  \foreach \n/\ang/\col in {#1}{
    \node[Xcf vertex,
      label={[Xcf label]\ang:{$\col$}}]
    at (\n) {};
  }
}

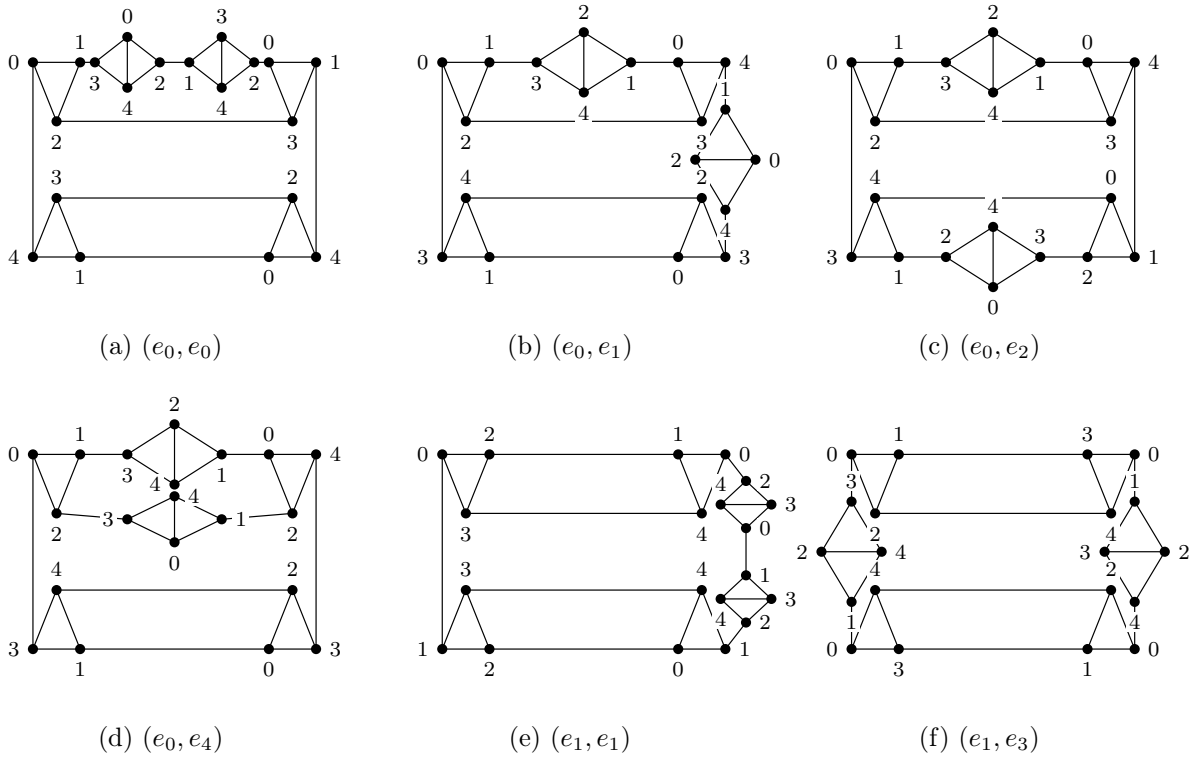
\begin{figure*}[h]
% Disable externalization locally, if the external library is loaded.
\ifdefined\tikzexternaldisable
  \tikzexternaldisable
\fi
\centering

% ================================================================
% (a) (e_0,e_0)
% ================================================================
\begin{subfigure}[t]{0.32\textwidth}
\centering

\begin{tikzpicture}[Xcf picture]
\Xcfbase

\foreach \n/\x/\y in {
  p01/-1.35/1.65,
  q01/-.25/1.65,
  x01/-.8/2.08,
  y01/-.8/1.22,
  p02/.25/1.65,
  q02/1.35/1.65,
  x02/.8/2.08,
  y02/.8/1.22
}{
  \coordinate (\n) at (\x,\y);
}

\draw (b0)--(p01);
\draw (p01)--(x01)--(q01)--(y01)--cycle;
\draw (x01)--(y01);
\draw (q01)--(p02);
\draw (p02)--(x02)--(q02)--(y02)--cycle;
\draw (x02)--(y02);
\draw (q02)--(a1);

\draw (b1)--(a2);
\draw (b2)--(a3);
\draw (b3)--(a0);
\draw (c0)--(c1);
\draw (c2)--(c3);

\Xcfcolors{
  a0/180/0,
  b0/90/1,
  c0/270/2,
  a1/90/0,
  b1/0/1,
  c1/270/3,
  a2/0/4,
  b2/270/0,
  c2/90/2,
  a3/270/1,
  b3/180/4,
  c3/90/3
}

\Xcfdiamond{p01}{-1.35}{1.65}{270}{3}
\Xcfdiamond{q01}{-.25}{1.65}{270}{2}
\Xcfdiamond{x01}{-.8}{2.08}{90}{0}
\Xcfdiamond{y01}{-.8}{1.22}{270}{4}

\Xcfdiamond{p02}{.25}{1.65}{270}{1}
\Xcfdiamond{q02}{1.35}{1.65}{270}{2}
\Xcfdiamond{x02}{.8}{2.08}{90}{3}
\Xcfdiamond{y02}{.8}{1.22}{270}{4}

\end{tikzpicture}
\caption{$(e_0,e_0)$}
\end{subfigure}
\hfill
%
% ================================================================
% (b) (e_0,e_1)
% ================================================================
\begin{subfigure}[t]{0.32\textwidth}
\centering

\begin{tikzpicture}[Xcf picture]
\Xcfbase

\foreach \n/\x/\y in {
  p0/-.8/1.65,
  q0/.8/1.65,
  x0/0/2.16,
  y0/0/1.14,
  p1/2.4/.85,
  q1/2.4/-.85,
  x1/2.91/0,
  y1/1.89/0
}{
  \coordinate (\n) at (\x,\y);
}

\draw (b0)--(p0);
\draw (p0)--(x0)--(q0)--(y0)--cycle;
\draw (x0)--(y0);
\draw (q0)--(a1);

\draw (b1)--(p1);
\draw (p1)--(x1)--(q1)--(y1)--cycle;
\draw (x1)--(y1);
\draw (q1)--(a2);

\draw (b2)--(a3);
\draw (b3)--(a0);
\draw (c0)--(c1);
\draw (c2)--(c3);

\Xcfcolors{
  a0/180/0,
  b0/90/1,
  c0/270/2,
  a1/90/0,
  b1/0/4,
  c1/270/3,
  a2/0/3,
  b2/270/0,
  c2/90/2,
  a3/270/1,
  b3/180/3,
  c3/90/4
}

\Xcfdiamond{p0}{-.8}{1.65}{270}{3}
\Xcfdiamond{q0}{.8}{1.65}{270}{1}
\Xcfdiamond{x0}{0}{2.16}{90}{2}
\Xcfdiamond{y0}{0}{1.14}{270}{4}

\Xcfdiamond{p1}{2.4}{.85}{90}{1}
\Xcfdiamond{q1}{2.4}{-.85}{270}{4}
\Xcfdiamond{x1}{2.91}{0}{0}{0}
\Xcfdiamond{y1}{1.89}{0}{180}{2}

\end{tikzpicture}
\caption{$(e_0,e_1)$}
\end{subfigure}
\hfill
%
% ================================================================
% (c) (e_0,e_2)
% ================================================================
\begin{subfigure}[t]{0.32\textwidth}
\centering

\begin{tikzpicture}[Xcf picture]
\Xcfbase

\foreach \n/\x/\y in {
  p0/-.8/1.65,
  q0/.8/1.65,
  x0/0/2.16,
  y0/0/1.14,
  p2/.8/-1.65,
  q2/-.8/-1.65,
  x2/0/-2.16,
  y2/0/-1.14
}{
  \coordinate (\n) at (\x,\y);
}

\draw (b0)--(p0);
\draw (p0)--(x0)--(q0)--(y0)--cycle;
\draw (x0)--(y0);
\draw (q0)--(a1);

\draw (b2)--(p2);
\draw (p2)--(x2)--(q2)--(y2)--cycle;
\draw (x2)--(y2);
\draw (q2)--(a3);

\draw (b1)--(a2);
\draw (b3)--(a0);
\draw (c0)--(c1);
\draw (c2)--(c3);

\Xcfcolors{
  a0/180/0,
  b0/90/1,
  c0/270/2,
  a1/90/0,
  b1/0/4,
  c1/270/3,
  a2/0/1,
  b2/270/2,
  c2/90/0,
  a3/270/1,
  b3/180/3,
  c3/90/4
}

\Xcfdiamond{p0}{-.8}{1.65}{270}{3}
\Xcfdiamond{q0}{.8}{1.65}{270}{1}
\Xcfdiamond{x0}{0}{2.16}{90}{2}
\Xcfdiamond{y0}{0}{1.14}{270}{4}

\Xcfdiamond{p2}{.8}{-1.65}{90}{3}
\Xcfdiamond{q2}{-.8}{-1.65}{90}{2}
\Xcfdiamond{x2}{0}{-2.16}{270}{0}
\Xcfdiamond{y2}{0}{-1.14}{90}{4}

\end{tikzpicture}
\caption{$(e_0,e_2)$}
\end{subfigure}

\par\vspace{1em}

% ================================================================
% (d) (e_0,e_4)
% ================================================================
\begin{subfigure}[t]{0.32\textwidth}
\centering

\begin{tikzpicture}[Xcf picture]
\Xcfbase

\foreach \n/\x/\y in {
  p0/-.8/1.65,
  q0/.8/1.65,
  x0/0/2.16,
  y0/0/1.14,
  p4/-.8/.55,
  q4/.8/.55,
  x4/0/.16,
  y4/0/.94
}{
  \coordinate (\n) at (\x,\y);
}

\draw (b0)--(p0);
\draw (p0)--(x0)--(q0)--(y0)--cycle;
\draw (x0)--(y0);
\draw (q0)--(a1);

\draw (c0)--(p4);
\draw (p4)--(x4)--(q4)--(y4)--cycle;
\draw (x4)--(y4);
\draw (q4)--(c1);

\draw (b1)--(a2);
\draw (b2)--(a3);
\draw (b3)--(a0);
\draw (c2)--(c3);

\Xcfcolors{
  a0/180/0,
  b0/90/1,
  c0/270/2,
  a1/90/0,
  b1/0/4,
  c1/270/2,
  a2/0/3,
  b2/270/0,
  c2/90/2,
  a3/270/1,
  b3/180/3,
  c3/90/4
}

\Xcfdiamond{p0}{-.8}{1.65}{270}{3}
\Xcfdiamond{q0}{.8}{1.65}{270}{1}
\Xcfdiamond{x0}{0}{2.16}{90}{2}
\Xcfdiamond{y0}{0}{1.14}{180}{4}

\Xcfdiamond{p4}{-.8}{.55}{180}{3}
\Xcfdiamond{q4}{.8}{.55}{0}{1}
\Xcfdiamond{x4}{0}{.16}{270}{0}
\Xcfdiamond{y4}{0}{.94}{0}{4}

\end{tikzpicture}
\caption{$(e_0,e_4)$}
\end{subfigure}
\hfill
%
% ================================================================
% (e) (e_1,e_1)
% ================================================================
\begin{subfigure}[t]{0.32\textwidth}
\centering

\begin{tikzpicture}[Xcf picture]
\Xcfbase

\foreach \n/\x/\y in {
  p11/2.75/1.2,
  q11/2.75/.4,
  x11/3.18/.8,
  y11/2.32/.8,
  p12/2.75/-.4,
  q12/2.75/-1.2,
  x12/3.18/-.8,
  y12/2.32/-.8
}{
  \coordinate (\n) at (\x,\y);
}

\draw (b1)--(p11);
\draw (p11)--(x11)--(q11)--(y11)--cycle;
\draw (x11)--(y11);
\draw (q11)--(p12);
\draw (p12)--(x12)--(q12)--(y12)--cycle;
\draw (x12)--(y12);
\draw (q12)--(a2);

\draw (b0)--(a1);
\draw (b2)--(a3);
\draw (b3)--(a0);
\draw (c0)--(c1);
\draw (c2)--(c3);

\Xcfcolors{
  a0/180/0,
  b0/90/2,
  c0/270/3,
  a1/90/1,
  b1/0/0,
  c1/270/4,
  a2/0/1,
  b2/270/0,
  c2/90/4,
  a3/270/2,
  b3/180/1,
  c3/90/3
}

\Xcfdiamond{p11}{2.75}{1.2}{0}{2}
\Xcfdiamond{q11}{2.75}{.4}{0}{0}
\Xcfdiamond{x11}{3.18}{.8}{0}{3}
\Xcfdiamond{y11}{2.32}{.8}{90}{4}

\Xcfdiamond{p12}{2.75}{-.4}{0}{1}
\Xcfdiamond{q12}{2.75}{-1.2}{0}{2}
\Xcfdiamond{x12}{3.18}{-.8}{0}{3}
\Xcfdiamond{y12}{2.32}{-.8}{270}{4}

\end{tikzpicture}
\caption{$(e_1,e_1)$}
\end{subfigure}
\hfill
%
% ================================================================
% (f) (e_1,e_3)
% ================================================================
\begin{subfigure}[t]{0.32\textwidth}
\centering

\begin{tikzpicture}[Xcf picture]
\Xcfbase

\foreach \n/\x/\y in {
  p1/2.4/.85,
  q1/2.4/-.85,
  x1/2.91/0,
  y1/1.89/0,
  p3/-2.4/-.85,
  q3/-2.4/.85,
  x3/-2.91/0,
  y3/-1.89/0
}{
  \coordinate (\n) at (\x,\y);
}

\draw (b1)--(p1);
\draw (p1)--(x1)--(q1)--(y1)--cycle;
\draw (x1)--(y1);
\draw (q1)--(a2);

\draw (b3)--(p3);
\draw (p3)--(x3)--(q3)--(y3)--cycle;
\draw (x3)--(y3);
\draw (q3)--(a0);

\draw (b0)--(a1);
\draw (b2)--(a3);
\draw (c0)--(c1);
\draw (c2)--(c3);

\Xcfcolors{
  a0/180/0,
  b0/90/1,
  c0/270/2,
  a1/90/3,
  b1/0/0,
  c1/270/4,
  a2/0/0,
  b2/270/1,
  c2/90/2,
  a3/270/3,
  b3/180/0,
  c3/90/4
}

\Xcfdiamond{p1}{2.4}{.85}{90}{1}
\Xcfdiamond{q1}{2.4}{-.85}{270}{4}
\Xcfdiamond{x1}{2.91}{0}{0}{2}
\Xcfdiamond{y1}{1.89}{0}{180}{3}

\Xcfdiamond{p3}{-2.4}{-.85}{270}{1}
\Xcfdiamond{q3}{-2.4}{.85}{90}{3}
\Xcfdiamond{x3}{-2.91}{0}{180}{2}
\Xcfdiamond{y3}{-1.89}{0}{0}{4}

\end{tikzpicture}
\caption{$(e_1,e_3)$}
\end{subfigure}

\caption{The six configurations, up to symmetry, obtained from
$X_{12}$ by inserting two diamonds. The number shown next to each
vertex is its assigned color.}
\label{fig:X12-two-diamonds}

\end{figure*}

\newpage

\bibliographystyle{plain}
\bibliography{referencias}

@article{LiLai2017,
  author  = {Hao Li and Hong-Jian Lai},
  title   = {3-dynamic coloring and list 3-dynamic coloring of {$K_{1,3}$}-free graphs},
  journal = {Discrete Applied Mathematics},
  volume  = {222},
  pages   = {166--171},
  year    = {2017},
  doi     = {10.1016/j.dam.2017.01.006}
}

@article{KangPark2019,
  author  = {Kang, Sungsik and Park, Boram},
  title   = {On incidence choosability of cubic graphs},
  journal = {Discrete Mathematics},
  volume  = {342},
  number  = {6},
  pages   = {1828--1837},
  year    = {2019}
}

@book{diestel_graph_2017,
  author    = {Diestel, Reinhard},
  title     = {Graph Theory},
  edition   = {5th},
  series    = {Graduate Texts in Mathematics},
  volume    = {173},
  publisher = {Springer},
  address   = {Berlin, Heidelberg},
  year      = {2017},
  doi       = {10.1007/978-3-662-53622-3},
  url       = {https://doi.org/10.1007/978-3-662-53622-3}
}

@book{bondy_graph_2008,
  author    = {Bondy, J. A. and Murty, U. S. R.},
  title     = {Graph Theory},
  series    = {Graduate Texts in Mathematics},
  volume    = {244},
  publisher = {Springer},
  address   = {New York, NY, USA},
  year      = {2008},
  isbn      = {978-1-84628-969-9},
  doi       = {10.1007/978-1-84628-970-5},
  url       = {https://link.springer.com/book/9781846289699}
}

@article{CHEN202224,
  title    = {Graph r-hued colorings—A survey},
  journal  = {Discrete Applied Mathematics},
  volume   = {321},
  pages    = {24-48},
  year     = {2022},
  issn     = {0166-218X},
  doi      = {https://doi.org/10.1016/j.dam.2022.06.003},
  url      = {https://www.sciencedirect.com/science/article/pii/S0166218X22001871},
  author   = {Ye Chen and Suohai Fan and Hong-Jian Lai and Murong Xu}
}

@article{LAI20061997,
  title    = {Conditional colorings of graphs},
  journal  = {Discrete Mathematics},
  volume   = {306},
  number   = {16},
  pages    = {1997-2004},
  year     = {2006},
  issn     = {0012-365X},
  doi      = {https://doi.org/10.1016/j.disc.2006.03.052},
  url      = {https://www.sciencedirect.com/science/article/pii/S0012365X06002561},
  author   = {Hong-Jian Lai and Jianliang Lin and Bruce Montgomery and Taozhi Shui and Suohai Fan}
}

@phdthesis{montgomery2001dynamic,
  author = {Montgomery, Bruce},
  title  = {Dynamic coloring of graphs},
  school = {West Virginia University},
  year   = {2001},
  type   = {Graduate Theses, Dissertations, and Problem Reports},
  number = {1397},
  url    = {https://researchrepository.wvu.edu/etd/1397}
}

@article{Oum2011, 
title={Perfect Matchings in Claw-free Cubic Graphs}, volume={18}, 
url={https://www.combinatorics.org/ojs/index.php/eljc/article/view/v18i1p62},
number={1},
journal={The Electronic Journal of Combinatorics},
author={Oum, Sang-il},
year={2011},
month={Mar.},
pages={P62}}

@article{CranstonKim2008,
  author  = {Daniel W. Cranston and Seog-Jin Kim},
  title   = {List-coloring the square of a subcubic graph},
  journal = {Journal of Graph Theory},
  volume  = {57},
  number  = {1},
  pages   = {65--87},
  year    = {2008},
  doi     = {10.1002/jgt.20273}
}

@article{Thomassen2018,
  author  = {Carsten Thomassen},
  title   = {The square of a planar cubic graph is 7-colorable},
  journal = {Journal of Combinatorial Theory, Series B},
  volume  = {128},
  pages   = {192--218},
  year    = {2018},
  doi     = {10.1016/j.jctb.2017.08.010}
}

@article{CamesVanBatenburgKang2019,
  author  = {Wouter Cames van Batenburg and Ross J. Kang},
  title   = {Squared Chromatic Number Without Claws or Large Cliques},
  journal = {Canadian Mathematical Bulletin},
  volume  = {62},
  number  = {1},
  pages   = {23--35},
  year    = {2019},
  doi     = {10.4153/CMB-2018-024-5}
}

@article{DeJoannisKangPastor2019,
  author  = {R{\'e}mi de Joannis de Verclos and Ross J. Kang and Lucas Pastor},
  title   = {Colouring Squares of Claw-free Graphs},
  journal = {Canadian Journal of Mathematics},
  volume  = {71},
  number  = {1},
  pages   = {113--129},
  year    = {2019},
  doi     = {10.4153/CJM-2017-029-9}
}

@article{BowlerEtAl2017,
  author  = {Nathan Bowler and Joshua Erde and Florian Lehner and
             Martin Merker and Max Pitz and Konstantinos Stavropoulos},
  title   = {A counterexample to Montgomery's conjecture on dynamic
             colourings of regular graphs},
  journal = {Discrete Applied Mathematics},
  volume  = {229},
  pages   = {151--153},
  year    = {2017},
  doi     = {10.1016/j.dam.2017.05.004}
}

\end{document}